\documentclass[11pt, reqno]{amsart}
\usepackage{fullpage, amsfonts,amsmath,amscd,amssymb}

\usepackage[enableskew]{youngtab}
\input xy
\xyoption{all}

\usepackage[mathscr,mathcal]{eucal}
\usepackage{colortbl}
\usepackage{amsfonts}
\usepackage{amssymb}
\usepackage{enumerate}
\newtheorem{theorem}{Theorem}[section]
\newtheorem{definition}[theorem]{Definition}

\newtheorem{lemma}[theorem]{Lemma}

\title{A cancellation-free formula for the Schur elements of the Ariki-Koike algebra}
\author{Maria Chlouveraki}
\address[M. Chlouveraki]{University of Edinburgh,
	School of Mathematics, JCMB, King's Buildings,
	Edinburgh, EH9 3JZ, UK
}
\email{maria.chlouveraki@ed.ac.uk}
\thanks{I would like to thank Iain Gordon and Stephen Griffeth for the conversations which led to the discovery of this pretty formula. In particular,  I am indebted to Stephen Griffeth for explaining to me the results of his paper \cite{DG}, which inspired this note.
I also gratefully acknowledge the support of the EPSRC through the grant EP/G04984X/1. }
\begin{document}
\maketitle

\section{Introduction}

Schur elements play a powerful role in the representation theory of symmetric algebras.
 In the case of the Ariki-Koike algebra, Schur elements are Laurent polynomials whose factors determine when Specht modules are projective irreducible and whether the algebra is semisimple.
 
Formulas for the Schur elements of the Ariki-Koike algebra have been independently obtained first by  
 Geck, Iancu and Malle \cite{GIM}, and later by Mathas \cite{Mat}.  The aim of this note is to give a cancellation-free formula for these polynomials (Theorem 5.1), so that their factors can be easily read and programmed.

\section{Partitions: definitions and notation}

A \emph{partition} $\lambda=(\lambda_1,\lambda_2,\lambda_3,\ldots)$ is a decreasing sequence of non-negative integers. We define the \emph{length of} $\lambda$ to be the smallest integer $\ell(\lambda)$ such that $\lambda_i=0$ for all $i>\ell(\lambda)$.
We write $|\lambda|:=\sum_{i \geq 1}\lambda_i$ and we say that $\lambda$ is a \emph{partition of} $m$, for some $m \in \mathbb{N}$, if $m=|\lambda|$. 
We set $n(\lambda):=\sum_{i \geq 1}  (i-1)\lambda_i$.

We define the set of nodes $[\lambda]$ of $\lambda$ to be the set
$$[\lambda]:=\{(i,j)\,\,|\,\, i\geq 1,\,\,1 \leq j \leq \lambda_i\}.$$
A node $x=(i,j)$ is called \emph{removable} if $[\lambda] \setminus \{(i,j)\}$ is still the set of nodes of a partition. Note that if $(i,j)$ is removable, then $j=\lambda_i$.

The \emph{conjugate partition} of $\lambda$ is the partition $\lambda'$ defined by
$$\lambda'_{k}:=\#\{i\,|\,i\geq 1 \text{ such that } \lambda_i\geq k\}.$$
Obviously, $\lambda_1'=\ell(\lambda)$. 
The set  of nodes of $\lambda'$ satisfies
$$(i,j) \in [\lambda'] \Leftrightarrow (j,i)\in [\lambda].$$
Note that if $(i,\lambda_i)$ is a removable node of $\lambda$, then
$\lambda_{\lambda_i}'=i.$
Moreover, we have $$n(\lambda)=\sum_{i \geq 1}  (i-1)\lambda_i= \frac{1}{2}\sum_{i \geq 1}(\lambda'_i-1)\lambda'_i.$$ 

Now, if $x=(i,j) \in [\lambda]$, we define the \emph{content} of $x$ to be the difference
$$\mathrm{cont}(x)=j-i.$$
The following lemma, whose proof is an easy combinatorial exercise (with the use of Young diagrams), relates the contents of the nodes of (the ``right rim'' of) $\lambda$ with the contents of the nodes of (the ``lower rim'' of) $\lambda'$.

\begin{lemma}\label{conj cont}
Let $\lambda=(\lambda_1,\lambda_2,\ldots)$ be a partition and let $k$ be an integer such that $1 \leq k \leq \lambda_1$. Let $q$ and $y$ be two indeterminates. Then we have
 $$\frac{1}{(q^{\lambda_1}y-1)}\cdot\left(\prod_{i=1}^{\lambda_k'} \frac{q^{\lambda_i-i+1}y-1}{q^{\lambda_i-i}y-1} \right)=\frac{1}{(q^{-\lambda'_{k}+k-1}y-1)}\cdot\left(\prod_{j=k}^{\lambda_1} \frac{q^{-\lambda'_j+j-1}y-1}{q^{-\lambda'_j+j}y-1} \right).
$$
\end{lemma}

Finally, if $x=(i,j) \in [\lambda]$ and $\mu$ is another partition, we define the \emph{generalized hook length of $x$ with respect to $\mu$} to be the integer:
            $$h_{i,j}^{\mu}:=\lambda_i-i+\mu'_j-j+1.$$ 
For $\mu=\lambda$, the above formula becomes the classical hook length formula (giving us the length of the hook of $\lambda$ that $x$ belongs to).

\section{The Ariki-Koike algebra}

Let $d$ and $r$ be positive integers and let $R$ be a commutative domain with $1$. Fix elements
$q,\,Q_0,\,\ldots,\,Q_{d-1}$ of $R$, and assume that $q$ is invertible in $R$. Set ${\bf q}:=(q;\,Q_0,\,\ldots,\,Q_{d-1})$.
The \emph{Ariki-Koike algebra} $\mathcal{H}_{d,r}$ is the unital associative $R$-algebra with generators $T_0,\,T_1,\,\ldots,\,T_{r-1}$ and relations:
\begin{center}
$(T_0 -Q_0) (T_0 -Q_1)\cdots(T_0 -Q_{d-1})=0,$\\
$(T_i-q)(T_i+1)=0$\,\, for $1\leq i \leq r-1$,\\
$T_0T_1T_0T_1=T_1T_0T_1T_0$,\\
$T_iT_{i+1}T_i=T_{i+1}T_iT_{i+1}$\,\, for $1\leq i \leq r-2$,\\
$T_iT_j=T_jT_i$\,\, for  $0\leq i <j \leq r-1$ with $j-i>1$.
\end{center}

The Ariki-Koike algebra $\mathcal{H}_{d,r}$ is a deformation of the group algebra of the complex reflection group $G(d,1,r)=(\mathbb{Z}/d\mathbb{Z})\wr \mathfrak{S}_r$. Ariki and Koike \cite{ArKo} have proved  that  $\mathcal{H}_{d,r}$  is a free $R$-module of rank $d^rr!=|G(d,1,r)|$. Moreover, Ariki \cite{Ar} has shown that, when $R$ is a field, $\mathcal{H}_{d,r}$ is (split) semisimple if and only if 
$$P({\bf q})=\prod_{i=1}^r(1+q+\cdots+q^{i-1}) \prod_{0 \leq s <t \leq d-1}\prod_{-r<k<r}(q^kQ_s-Q_t)$$
is a non-zero element of $R$. 

A $d$-\emph{partition} of  $r$ is an ordered $d$-tuple $\lambda=(\lambda^{(0)},\lambda^{(1)},\ldots,\lambda^{(d-1)})$  of partitions $\lambda^{(s)}$ such that $\sum_{s=0}^{d-1}|\lambda^{(s)}|=r$.
Let us denote by $\mathcal{P}(d,r)$ the set of $d$-partitions of $r$.
In the semisimple case, Ariki and Koike \cite{ArKo}  constructed an irreducible $\mathcal{H}_{d,r}$-module $S^\lambda$, called a \emph{Specht module}, for each $d$-partition $\lambda$ of $r$. Further, they showed that 
$\{S^\lambda \,|\,\lambda \in \mathcal{P}(d,r)\}$
is a complete set of pairwise non-isomorphic irreducible $\mathcal{H}_{d,r}$-modules. We denote by $\chi^\lambda$ the character of the Specht module $S^\lambda$.

Now,  there exists a linear form $\tau:\mathcal{H}_{d,r} 
\rightarrow R$ which was introduced by Bremke and Malle in \cite{BreMa}, and was proved to be symmetrizing by Malle and Mathas in \cite{MaMa} whenever all $Q_i$'s are invertible in $R$. An explicit description of this form can be found in any of these two articles. Following Geck's results on symmetrizing forms \cite{Gehab}, we obtain the following definition for the Schur elements associated to the irreducible representations of $\mathcal{H}_{d,r}$.

\begin{definition}\label{Schur}
Suppose that $R$ is a field and that $P({\bf q}) \neq 0$. The \emph{Schur elements} of 
$\mathcal{H}_{d,r}$ are the elements $s_\lambda({\bf q})$ of $R$ such that
$$\tau = \sum_{\lambda \in \mathcal{P}(d,r)} \frac{1}{s_\lambda({\bf q})}\chi^\lambda. $$
\end{definition}

Schur elements play a powerful role in the representation theory of $\mathcal{H}_{d,r}$, as illustrated by the following result (cf.~\cite[Theorem 7.4.7]{GePf}, \cite[Lemme 2.6]{MaRou}).

%?????

\begin{theorem}
Suppose that $R$ is a field.
If $s_\lambda({\bf q}) \neq 0$, then the Specht module $S^\lambda$ is irreducible. Moreover,
the algebra $\mathcal{H}_{d,r}$ is semisimple if and only if $s_\lambda({\bf q}) \neq 0$ for all
$\lambda \in \mathcal{P}(d,r)$.
\end{theorem}

\section{Formulas for the Schur elements of the Ariki-Koike algebra}

The Schur elements of the Ariki-Koike algebra $\mathcal{H}_{d,r}$ have been independently calculated first by Geck, Iancu and Malle \cite{GIM}, and later by Mathas \cite{Mat}. 
From now on, for all $m \in \mathbb{N}$, let $[m]_q:=(q^m-1)/(q-1)=q^{m-1}+q^{m-2}+\cdots+q+1$.
The formula given by Mathas does not demand extra notation and is the following:

\begin{theorem}\label{Mathas} Let $\lambda=(\lambda^{(0)},\lambda^{(1)},\ldots,\lambda^{(d-1)})$ be a $d$-partition of $r$. Then
$$s_\lambda({\bf q})=(-1)^{r(d-1)} (Q_0Q_1\cdots Q_{d-1})^{-r}
q^{-\alpha({\lambda}')} \prod_{s=0}^{d-1}  \prod_{(i,j) \in [\lambda^{(s)}]}
 Q_s[{{h_{i,j}^{\lambda^{(s)}}}}]_q \cdot
 \prod_{0 \leq s< t \leq d-1} X_{st}^{\lambda} ,$$
where
$$\alpha({\lambda}')=\frac{1}{2}\sum_{s=0}^{d-1}\sum_{i \geq 1}
(\lambda^{(s)'}_i-1)\lambda^{(s)'}_i$$
and
$$ X_{st}^{\lambda}=  
 \prod_{(i,j) \in [\lambda^{(t)}]}(q^{j-i}Q_t-Q_s) 
\cdot
 \prod_{(i,j) \in [\lambda^{(s)}]}\left((q^{j-i}Q_s-q^{\lambda_1^{(t)}}Q_t) 
\prod_{k=1}^{\lambda_1^{(t)}}
\frac{q^{j-i}Q_s-q^{k-1-\lambda_k^{(t)'}}Q_t}{q^{j-i}Q_s-q^{k-\lambda_k^{(t)'}}Q_t}\right).  
$$
\end{theorem}
$ $

The formula by Geck, Iancu and Malle is more symmetric, and describes the Schur elements in terms of \emph{beta numbers}. If $\lambda=(\lambda^{(0)},\lambda^{(1)},\ldots,\lambda^{(d-1)})$ is a $d$-partition of $r$, then the \emph{length of} $\lambda$ is $\ell(\lambda)=\mathrm{max}\{\ell(\lambda^{(s)})\,|\,0\leq s \leq d-1\}$. 
 Fix an integer $L$ such that $L \geq \ell(\lambda)$. The $L$-\emph{beta numbers} for $\lambda^{(s)}$ are the integers $\beta_i^{(s)}=\lambda_i^{(s)}+L-i$ for $i = 1, \ldots , L$. 
Set $B^{(s)} = \{\beta_1^{(s)}, \ldots, \beta_L^{(s)}\}$ for $s=0,\ldots,d-1$. The matrix $B=(B^{(s)})_{0\leq s \leq d-1}$ is called the $L$-\emph{symbol} of $\lambda$.

\begin{theorem}\label{sym}
Let  $\lambda=(\lambda^{(0)},\ldots,\lambda^{(d-1)})$ be a $d$-partition of $r$
with $L$-symbol $B=(B^{(s)})_{0\leq s \leq d-1}$, where $L \geq \ell(\lambda)$.  Let $a_{L}:=r(d-1)+\binom{ d}{ 2}\binom{ L}{ 2}$ and $b_{L}:=dL(L-1)(2dL-d-3)/12$. Then  $$s_\lambda({\bf q})=(-1)^{a_{L}} x^{b_{L}}(q-1)^{-r}(Q_0Q_1\ldots Q_{d-1})^{-r}\nu_\lambda/ \delta_\lambda,$$ where
$$
\nu_\lambda=
\prod_{0\leq s<t\leq d-1}(Q_s-Q_t)^L\prod_{0 \leq s,\,t \leq d-1}\prod_{b_s \in B^{(s)}}\prod_{1 \leq k \leq b_s} 
(q^kQ_s-Q_t)$$
and
$$\delta_\lambda=\prod_{0\leq s< t \leq d-1}\prod_{(b_s,b_t) \in B^{(s)}\times B^{(t)}}(q^{b_s}Q_s-q^{b_t}Q_t) \prod_{0 \leq s \leq d-1} \prod_{1 \leq i < j \leq L}(q^{b_i^{(s)}}Q_s-q^{b_j^{(s)}}Q_s).
$$
\end{theorem}
$ $
As the reader may see, in both formulas above, the factors of $s_\lambda({\bf q})$ are not obvious.
Hence, it is not obvious for which values of ${\bf q}$ the Schur element $s_\lambda({\bf q})$ becomes zero.

\section{A cancellation-free formula}

In this section, we will give a  cancellation-free formula for the Schur elements of $\mathcal{H}_{d,r}$. This formula is also symmetric.

Let $\lambda=(\lambda^{(0)},\lambda^{(1)},\ldots,\lambda^{(d-1)})$ be a $d$-partition of $r$. 
The multiset $(\lambda_i^{(s)})_{0 \leq s \leq d-1,\,i\geq 1}$ is a composition of $r$ (\emph{i.e.}
a multiset of non-negative integers whose sum is equal to $r$). By reordering the elements of this composition, we obtain a partition of $r$. We denote this partition by $\bar{\lambda}$.
(e.g., if $\lambda=((4,1),\emptyset,(2,1))$, then $\bar{\lambda}=(4,2,1,1)$).

\begin{theorem} Let $\lambda=(\lambda^{(0)},\lambda^{(1)},\ldots,\lambda^{(d-1)})$ be a $d$-partition of $r$. 
Then 
\begin{equation}\label{pretty}
s_\lambda({\bf q})=(-1)^{r(d-1)}q^{-n(\bar{\lambda})}(q-1)^{-r} \prod_{s=0}^{d-1}
\prod_{(i,j) \in [\lambda^{(s)}]} \prod_{t=0}^{d-1}  (q^{h_{i,j}^{\lambda^{(t)}}}Q_sQ_t^{-1}-1).
\end{equation}
%Note that whenever $s=t$ in the above formula, we have that $(q-1)$
%divides  $(q^{h_{i,j}^{\lambda^{(s)}}}-1)$. 
Since the total number of nodes in $\lambda$ is equal to $r$, the above formula can be rewritten as follows:
\begin{equation}\label{claim}
s_\lambda({\bf q})=(-1)^{r(d-1)}q^{-n(\bar{\lambda})}
 \prod_{0 \leq s \leq d-1} \prod_{(i,j) \in [\lambda^{(s)}]}\left( [{{h_{i,j}^{\lambda^{(s)}}}}]_q
 \prod_{0 \leq t \leq d-1,\, t\neq s} (q^{h_{i,j}^{\lambda^{(t)}}}Q_sQ_t^{-1}-1)\right).
 \end{equation}
\end{theorem}
$ $

We will now proceed to the proof of the above result.
Following Theorem \ref{Mathas}, we have that

$$s_\lambda({\bf q})=(-1)^{r(d-1)} (Q_0Q_1\cdots Q_{d-1})^{-r}
q^{-\alpha({\lambda}')} \prod_{s=0}^{d-1}  \prod_{(i,j) \in [\lambda^{(s)}]}
 Q_s[{{h_{i,j}^{\lambda^{(s)}}}}]_q \cdot
 \prod_{0 \leq s< t \leq d-1} X_{st}^{\lambda} ,$$
where
$$\alpha({\lambda}')=\frac{1}{2}\sum_{s=0}^{d-1}\sum_{i \geq 1}
(\lambda^{(s)'}_i-1)\lambda^{(s)'}_i$$
and
$$ X_{st}^{\lambda}=  
 \prod_{(i,j) \in [\lambda^{(t)}]}(q^{j-i}Q_t-Q_s) 
\cdot
 \prod_{(i,j) \in [\lambda^{(s)}]}\left((q^{j-i}Q_s-q^{\lambda_1^{(t)}}Q_t) 
\prod_{k=1}^{\lambda_1^{(t)}}
\frac{q^{j-i}Q_s-q^{k-1-\lambda_k^{(t)'}}Q_t}{q^{j-i}Q_s-q^{k-\lambda_k^{(t)'}}Q_t}\right).  
$$

The following lemma relates the terms $q^{-n(\bar{\lambda})}$ and $q^{-\alpha({\lambda}')}$ .

\begin{lemma}\label{lemma q}
Let $\lambda$ be a $d$-partition of $r$. We have that
$$\alpha(\lambda')+\sum_{0\leq s <t\leq d-1}\sum_{i\geq 1}\lambda^{(s)'}_i\lambda^{(t)'}_i=
n(\bar{\lambda}).$$
\end{lemma}
\begin{proof}
%We have
%$$\begin{array}{ccc}
%\bar{\lambda}^{'}_1 &= &\sum_{s=0}^{d-1}\lambda_1^{(s)'}\\ 
%\bar{\lambda}^{'}_2 &= &\sum_{s=0}^{d-1}\lambda_2^{(s)'}\\ 
%\vdots& &\vdots\\
%\bar{\lambda}^{'}_i &= &\sum_{s=0}^{d-1}\lambda_i^{(s)'}
%\end{array}$$
%for all $i$. Therefore,
Following the definition of the conjugate partition, we have
$\bar{\lambda}^{'}_i = \sum_{s=0}^{d-1}\lambda_i^{(s)'},$
for all $i \geq 1$. Therefore,
$$n(\bar{\lambda})=\frac{1}{2}\sum_{i \geq 1}(\bar{\lambda}^{'}_i-1)\bar{\lambda}^{'}_i=
\frac{1}{2}\sum_{i \geq 1}\left(\left(\sum_{s=0}^{d-1}\lambda_i^{(s)'}-1\right)\cdot\sum_{s=0}^{d-1}\lambda_i^{(s)'}\right)$$
$$=\frac{1}{2}\sum_{i \geq 1}
\left( \sum_{0\leq s <t\leq d-1}2\cdot\lambda^{(s)'}_i\lambda^{(t)'}_i 
+\sum_{s=0}^{d-1}{\lambda_i^{(s)'}}^2-\sum_{s=0}^{d-1}{\lambda_i^{(s)'}}\right)$$
$$= \sum_{0\leq s <t\leq d-1}\sum_{i \geq 1}\lambda^{(s)'}_i\lambda^{(t)'}_i +
\frac{1}{2} \sum_{s=0}^{d-1}\sum_{i \geq 1}(\lambda^{(s)'}_i-1)\lambda^{(s)'}_i
=\sum_{0\leq s <t\leq d-1}\sum_{i \geq 1}\lambda^{(s)'}_i\lambda^{(t)'}_i +\alpha(\lambda')
$$

\end{proof}

Hence, to prove Equality (\ref{claim}), it is enough to show that, for all $0\leq s <t\leq d-1$,
\begin{equation}\label{X}
X_{st}^{\lambda}=q^{-\sum_{i\geq 1}\lambda^{(s)'}_i\lambda^{(t)'}_i} 
Q_s^{| \lambda^{(t)}|} Q_t^{|\lambda^{(s)}|}
\prod_{(i,j) \in [\lambda^{(s)}]}
 (q^{h_{i,j}^{\lambda^{(t)}}}Q_sQ_t^{-1}-1)\cdot 
\prod_{(i,j) \in [\lambda^{(t)}]}
 (q^{h_{i,j}^{\lambda^{(s)}}}Q_tQ_s^{-1}-1). 
\end{equation}

We will proceed by induction on the number of nodes of $\lambda^{(s)}$. We do not need to do the same for $\lambda^{(t)}$, because the symmetric formula for the Schur elements given by Theorem \ref{sym} implies the following: if $\mu$ is the multipartition obtained from $\lambda$ by exchanging 
$\lambda^{(s)}$ and $\lambda^{(t)}$, then
$$X_{st}^{\lambda}(Q_s,Q_t)=X_{st}^{\mu}(Q_t,Q_s).$$

If $\lambda^{(s)}=\emptyset$, then
$$X_{st}^{\lambda}=\prod_{(i,j) \in [\lambda^{(t)}]}(q^{j-i}Q_t-Q_s)=Q_s^{| \lambda^{(t)}|}
\prod_{(i,j) \in [\lambda^{(t)}]}(q^{j-i}Q_tQ_s^{-1}-1)=
%$$\prod_{(i,j) \in [\lambda^{(t)}]}(q^{j-i}Q_tQ_s^{-1}-1)=
Q_s^{| \lambda^{(t)}|}\prod_{i=1}^{\lambda^{(t)'}_1} \prod_{j=1}^{\lambda^{(t)}_i}
(q^{j-i}Q_tQ_s^{-1}-1)$$ $$=
Q_s^{| \lambda^{(t)}|}\prod_{i=1}^{\lambda^{(t)'}_1} \prod_{j=1}^{\lambda^{(t)}_i}
(q^{\lambda^{(t)}_i-j+1-i}Q_tQ_s^{-1}-1)=
Q_s^{| \lambda^{(t)}|}\prod_{(i,j) \in [\lambda^{(t)}]}
 (q^{h_{i,j}^{\lambda^{(s)}}}Q_tQ_s^{-1}-1), 
$$
as required.
\\

Now, assume that our assertion holds when $\#[\lambda^{(s)}] \in \{0,1,2,\ldots,N-1\}$. We want to show that it also holds when $\#[\lambda^{(s)}]=N \geq 1$.
If $\lambda^{(s)}\neq \emptyset$, then there exists $i$ such that $(i,\lambda_i^{(s)})$ is a removable node of $\lambda^{(s)}$. Let $\nu$ be the multipartition defined by
$$\nu^{(s)}_i:=\lambda^{(s)}_i-1,\,\,\, \nu^{(s)}_j:=\lambda^{(s)}_j \text{ for all } j\neq i,\,\,\, \nu^{(t)}:=\lambda^{(t)} \text{ for all } t\neq s.$$
Then $[\lambda^{(s)}]=[\nu^{(s)}] \cup \{(i,\lambda^{(s)}_i)\}$. Since Equality $(\ref{X})$ holds for
$X_{st}^{\nu}$ and 
$$X_{st}^{\lambda}=X_{st}^{\nu} \cdot \left((q^{\lambda^{(s)}_i-i}Q_s-q^{\lambda_1^{(t)}}Q_t) 
\prod_{k=1}^{\lambda_1^{(t)}}
\frac{q^{\lambda^{(s)}_i-i}Q_s-q^{k-1-\lambda_k^{(t)'}}Q_t}{q^{\lambda^{(s)}_i-i}Q_s-q^{k-\lambda_k^{(t)'}}Q_t}\right),$$
it is enough to show that (to simplify notation, from now on set $\lambda:=\lambda^{(s)}$ and $\mu:=\lambda^{(t)}$):
\begin{equation}\label{add1}
(q^{\lambda_i-i}Q_s-q^{\mu_1}Q_t)
\prod_{k=1}^{\mu_1}
\frac{q^{\lambda_i-i}Q_s-q^{k-1-\mu_k'}Q_t}{q^{\lambda_i-i}Q_s-q^{k-\mu_k'}Q_t}
=
q^{-\mu_{\lambda_i}'}
Q_t  (q^{\lambda_i-i+\mu_{\lambda_i}'-\lambda_i+1}Q_sQ_t^{-1}-1)\cdot A \cdot B,
\end{equation}
where
$$A:=
\prod_{k=1}^{\lambda_i-1}
\frac{q^{\lambda_i-i+\mu_k'-k+1}Q_sQ_t^{-1}-1}
{q^{\lambda_i-i+\mu_k'-k}Q_sQ_t^{-1}-1}$$
and
$$B:=\prod_{k=1}^{\mu_{\lambda_i}'}
\frac{q^{\mu_k-k+\lambda_{\lambda_i}'-\lambda_i+1}Q_tQ_s^{-1}-1}
{q^{\mu_k-k+\lambda_{\lambda_i}'-\lambda_i}Q_tQ_s^{-1}-1}.$$
\\
Note that, since $(i,\lambda_i)$ is a removable node of $\lambda$, we have $\lambda_{\lambda_i}'=i$.
We have that
$$A=q^{\lambda_i-1}
\prod_{k=1}^{\lambda_i-1}
\frac{q^{\lambda_i-i}Q_s-q^{k-1-\mu_k'}Q_t}
{q^{\lambda_i-i}Q_s-q^{k-\mu_k'}Q_t}.$$
Moreover, by Lemma \ref{conj cont}, for $y=q^{i-\lambda_i}Q_tQ_s^{-1},$ we obtain that
$$B=\frac{(q^{\mu_1+i-\lambda_i}Q_tQ_s^{-1}-1)}
{(q^{-\mu_{\lambda_i}'+\lambda_i-1+i-\lambda_i}Q_tQ_s^{-1}-1)}\cdot
\left(\prod_{k=\lambda_i}^{\mu_{1}}
\frac{q^{-\mu_k'+k-1+i-\lambda_i}Q_tQ_s^{-1}-1}
{q^{-\mu_k'+k+i-\lambda_i}Q_tQ_s^{-1}-1}\right),$$
\emph{i.e.,}
$$B=Q_t^{-1}q^{\mu_{\lambda_i}'-\lambda_i+1}
\frac{(q^{\lambda_i-i}Q_s-q^{\mu_1}Q_t)}
{(q^{\mu_{\lambda_i}'-\lambda_i+1+\lambda_i-i}Q_sQ_t^{-1}-1)}\cdot
\left(\prod_{k=\lambda_i}^{\mu_{1}}
\frac{q^{\lambda_i-i}Q_s-q^{k-1-\mu_k'}Q_t}{q^{\lambda_i-i}Q_s-q^{k-\mu_k'}Q_t}\right).$$
\\
Hence, Equality (\ref{add1}) holds.

\end{document}